\documentclass[12pt]{article}
\usepackage{fullpage}
\usepackage{amsmath}
\usepackage{amssymb}
\usepackage{amsthm}

\newtheorem{theorem}{Theorem}
\newtheorem{cor}{Corollary}
\newtheorem{prop}{Proposition}
\newtheorem{lemma}{Lemma}

\def\ord{ord}
\def\eps{\varepsilon}
\def\F{\mathbb{F}}
\def\Z{\mathbb{Z}}

\title{Explicit upper bounds on the least primitive root}
\author{Kevin J. McGown\\ Department of Mathematics and Statistics\\ California State University, Chico\\
School of Science\\The University of New South Wales, Canberra\\
kmcgown@csuchico.edu \and Tim Trudgian\footnote{Supported by Australian Research Council Future Fellowship FT160100094.}\\ School of Science\\ The University of New South Wales, Canberra\\
  t.trudgian@adfa.edu.au}

\begin{document}

\maketitle
\begin{abstract}
\noindent
We give a method for producing explicit bounds on $g(p)$, the least primitive root modulo $p$.
Using our method we show that
$g(p)<2r\,2^{r\omega(p-1)}\,p^{\frac{1}{4}+\frac{1}{4r}}$ for $p>10^{56}$
where $r\geq 2$ is an integer parameter.
This result beats existing bounds that rely on explicit versions of the Burgess inequality.
Our main result allows one to derive bounds of differing shapes for various ranges of $p$.
For example, our method also allows us to show that $g(p)<p^{5/8}$ for all $p\geq 10^{22}$
and $g(p)<p^{1/2}$ for $p\geq 10^{56}$.
\end{abstract}

\section{Introduction}\label{intro}

Let $p$ be an odd prime and let $g(p)$ denote the least primitive root modulo $p$.
Giving an upper bound on $g(p)$ is a classical problem that has received much attention.
The best known asymptotic bound, due to Burgess \cite{Burgess62}, is
\mbox{$g(p)\ll p^{1/4+\eps}$}.

Consider the indicator function
\begin{align*}
f(n)=
\begin{cases}
  1 & \text{if $n$ is a primitive root modulo $p$}\\
  0 & \text{otherwise}\,.  
  \end{cases}
\end{align*}
It is well-known (going back at least to Landau) that
\[
 f(n)=\frac{\phi(p-1)}{p-1}\sum_{d|p-1}\frac{\mu(d)}{\phi(d)}\sum_{\ord(\chi)=d}\chi(n)
 \,,
\]
where the inner sum is taken over all $\phi(d)$ Dirichlet characters $\chi$ with multiplicative order~$d$.
Let $r\geq 2$ be an integer.  Using an explicit version of the Burgess inequality of the form
\begin{equation}\label{slice}
  \left|\sum_{n\leq H}\chi(n)\right|\leq C(r)H^{1-\frac{1}{r}}p^{\frac{r+1}{4r^2}}(\log p)^{\frac{1}{2r}}
\end{equation}
one can show\footnote{This is implicit in the work of Cohen, Oliveira e Silva, and Trudgian \cite[p.\ 264]{Cohen}.} that
\begin{equation}\label{dice}
  g(p)\leq C(r)^r\,2^{r\omega(p-1)} p^{\frac{1}{4}+\frac{1}{4r}}(\log p)^{\frac{1}{2}}
\end{equation}
for $p$ sufficiently large. 
For example, when $p\geq 10^{15}$,
Trevi\~no \cite{Trevino} obtains the following constants in (\ref{slice}) which are the best known.


\begin{table}[ht]
 \centering
 \caption{Trevi\~{n}o's values for $C(r)$ and $C(r)^{r}$ in (\ref{slice}) and (\ref{dice})}\medskip
\begin{tabular}{c|c|c}
$r$ & $C(r)$ & $C(r)^r$\\
\hline
2 & 3.5851 & 12.8530\\
3 & 2.5144& 15.8966\\
4 & 2.1258& 20.4216\\
5 & 1.9231& 26.3033 \\
6 & 1.7959& 33.5501 \\
7 & 1.7066& 42.1621\\
8 & 1.6384& 51.9230 \\
9 & 1.5857 & 63.3855\\
10 & 1.5410& 75.5139\\
\end{tabular}
\end{table}

We prove the following explicit upper bound for $g(p)$
that not only improves these constants,
but also removes the log term completely.

\begin{theorem}\label{T:1}
For all $p\geq 10^{56}$ and for any integer $r\geq 2$, we have
\[
  g(p)<2r\,2^{r\omega(p-1)}\,p^{\frac{1}{4}+\frac{1}{4r}}
  \,.
\]
\end{theorem}

The novelty in our proof is that we do not use the Burgess inequality directly,
although we will use many ingredients that go into its proof.
Indeed, although the exponent on the $\log p$ term in the Burgess inequality
has been recently improved (see~\cite{Kerr}) it does not seem
possible at present to remove it completely.

Theorem~\ref{T:1} is a consequence of our main theorem (see Theorem~\ref{T:3}) which is more flexible but more complicated to state.
Moreover, it holds for all $p$ and allows one to derive bounds of differing shapes for various ranges of $p$.
For example, one corollary is the following easy-to-state explicit bound.
\begin{cor}\label{T:2}
For all $p\geq 10^{22}$, we have $g(p)<p^{5/8}$.
\end{cor}
We note that proving good results of the form $g(p) \leq p^{\alpha}$ for all primes $p$ appears to be difficult.
Cohen, Oliveira e Silva and Trudgian \cite{Cohen} proved that $g(p) \leq 5.2p^{0.99}$. This was improved to $p^{0.96}$ by Cohen and Trudgian \cite{CohenT}, to $p^{0.88}$ by Hunter \cite{Hunter}, and to $p^{0.68}$ by Pretorius \cite{Pretorius}.  These latter results use numerically efficient versions of the P\'{o}lya--Vinogradov inequality (see, e.g., \cite{Sound})
together with some amount of computation.
In this investigation, we will not pursue a result that holds for all~$p$.

Grosswald \cite{Grosswald81} conjectured that $g(p)<\sqrt{p}-2$ for all $p>409$;
he showed that this implies that for $p>2$, the principal congruence subgroup
$\Gamma(p)$ can always be freely generated by the matrix $[1, p; 0, 1]$ and
$p(p-1)(p+1)/12$ additional \emph{hyperbolic} matrices.
Cohen, Oliveira e Silva, and Trudgian \cite{Cohen} proved that this holds
except possibly when $p\in(2.5\cdot 10^{15}, 10^{71})$. This has recently been improved by Jarso, Kerr, and Shparlinski \cite{Jarso} who showed that $g(p)<\sqrt{p}-2$ for all $p\geq 10^{65}$. We improve this further in the following result.

\begin{cor}\label{lonely}
When $p\geq 10^{56}$, we have $g(p)<(0.999)p^{1/2}< \sqrt{p} -2$.
\end{cor}

Finally, we also give a sieved version of Theorem~\ref{T:1},
which we believe to be useful in applications.
\begin{theorem}\label{Odd fellow}
Let $p\geq 10^{56}$ be prime.
Let $e$ be an even divisor of $p-1$ and let $p_1,\dots,p_s$ be the primes dividing $p-1$ that do not divide $e$.
Set $\delta=1-\sum_{i=1}^sp_i^{-1}$.  Provided $\delta>0$, we have
\[
  g(p)<
  2\,r\left(
  \left(2+\frac{s-1}{\delta}\right)
  2^{\omega(p-1)-s}
  \right)^r
  \,p^{\frac{1}{4}+\frac{1}{4r}}.
\]
\end{theorem}


Throughout this paper, $p$ will denote an odd prime.
We will write $\omega=\omega(p-1)$ for the number of distinct prime factors of $p-1$.
We write $\phi(n)$ to denote the Euler totient function, $\mu(n)$ to denote the M\"obius function, and $\theta(n)$ to denote the multiplicative function $\theta(n)=\phi(n)/n$.
The notation $f(n)$ will always denote the primitive root indicator function.

In \S\ref{king} we collect some preliminary results, in~\S\ref{queen} we prove our main result in Theorem \ref{T:3} and in~\S\ref{prince} we
flesh out the consequences of this result which includes the proofs of the results mentioned in~\S\ref{intro}.

\section{Preparations}\label{king}

We require three primary ingredients.  The first is an upper bound on a certain character sum that draws its strength from the Weil bound, the second is an estimate on the number of integer points in a special collection of intervals, and the third is a combinatorial sieve.

\subsection{A character sum estimate}

Define the sum
\[
  S_\chi(p,h,r):=\sum_{x\in\F_p}\left|\sum_{n=0}^{h-1}\chi(x+n)\right|^{2r}
  \,.
\]

\begin{lemma}\label{stirling}
We have
\[
  \left(\frac{2r}{e}\right)^r
  <
  \frac{(2r)!}{2^r r!}<
  \sqrt{2}
  \left(\frac{2r}{e}\right)^r.
\]
\end{lemma}
\begin{proof}
Apply an explicit version of Stirling's formula (see, e.g., \cite{Robbins}).
\end{proof}

The first part of the following proposition is due to Trevi\~no,
following Burgess, Norton, and Booker (see~\cite{TrevinoJNT});
the second part is a refinement in a special case.

\begin{prop}\label{prop:weil}
Let $\chi$ be a non-principal Dirichlet character modulo $p$.
Let $r,h\in\Z^+$.  Then
\item
\[
  S_\chi(p,h,r)\leq \frac{(2r)!}{2^r r!}ph^r + (2r-1)p^{1/2}h^{2r}
\]
and
\begin{equation}\label{booze}
S_\chi(p,h,2)
  <
  \begin{cases}
  (3h^2-2h)p + 2(h^4-3h^2+2h)p^{1/2} & \text{ if $\ord(\chi)=2$}\\
    (2h^2-h)p + 3(h^4-2h^2+h)p^{1/2} & \text{ if $\ord(\chi)>2$}\,.
    \end{cases}
\end{equation}
\end{prop}
\begin{proof}
Without loss of generality, we may assume that \mbox{$h<p$}. 
Additionally, we notice that if $r\geq (e/2)h$, then the proposition
is trivial since we would have
(in light of Lemma~\ref{stirling})
$$
S(\chi,h,r)\leq h^{2r}p\leq \left(\frac{2r}{e}\right)^r h^r p< \frac{(2r)!}{2^r r!}h^r p\,.
$$
Hence we may assume that $r< (e/2)h$.

To begin, we observe that
$$
  S(\chi,h,r)
  =
  \sum_{1\leq m_1,\dots,m_{2r}\leq h}\;\;
  \sum_{x=0}^{p-1}
  \chi(x+m_1)\dots\chi(x+m_r)\overline{\chi}(x+m_{r+1})\dots\overline{\chi}(x+m_{2r})
  \,.
$$
Define
$$
  \mathcal{M}:=\{\mathbf{m}=(m_1,\dots,m_{2r})\mid 1\leq m_1,\dots,m_{2r}\leq h\}
$$
and
$$
  F_\mathbf{m}(x)
  =
  (x+m_1)\dots(x+m_r)(x+m_{r+1})^{n-1}(x+m_{2r})^{n-1},
$$
where $n$ denotes the order of $\chi$.
We can then rewrite the above as
$$
  S(\chi,h,r)=
  \sum_{\mathbf{m}\in\mathcal{M}}
  \sum_{x\in\F_p}
  \chi(F_\mathbf{m}(x))
  \,.
$$
If $F_\mathbf{m}(x)$ is not an $n$-th power mod $p$,
then we can invoke the Weil bound (see, for example, Theorem~11.23 of~\cite{IK})
to obtain
\begin{equation}\label{weilbound}
  \left|
  \sum_{x\in\F_p}
  \chi(F_\mathbf{m}(x))
  \right|
  \leq
  (2r-1)\sqrt{p}
  \,.
\end{equation}
Otherwise, we accept the trivial bound of $p$.

It remains to count the number of exceptions --- that is, the number of $\mathbf{m}\in\mathcal{M}$ such that $F_\mathbf{m}(x)$ is an $n$-th power mod $p$.
If $n=2$, it is easy to see that the number of exceptions is bounded above
by $(2r-1)(2r-3)\dots(3)(1)=(2r)!/(2^r r!)$ simply by pairing each $m_j$ with a duplicate.
When $n>2$, the counting problem is much more difficult.
Trevi\~no (see Lemma 2.1 of~\cite{TrevinoJNT}) shows that the number of exceptions is bounded above by the quantity
\[
  c_r(h,n)=\sum_{d=0}^{\lfloor\frac{r}{n}\rfloor}
  \left(\frac{r!}{d!(n!)^d}\right)^2\frac{h^{r-(n-2)d}}{(r-nd)!}
  \,;
\]
moreover, under the condition $r\leq 9h$, he shows that $c_r(h,n)$
is a decreasing function of $n$ and hence $c_r(h,n)\leq c_r(h,2)=(2r)!/(2^r r!)h^r$.
But since we have $r< (e/2)h$ in the context of our proof, this condition
is automatic.

Specializing to $r=2$, our polynomial becomes
$$
  F_\mathbf{m}(x)=(x+m_1)(x+m_2)(x+m_{3})^{n-1}(x+m_4)^{n-1}
  \,.
$$
First, consider the case where $n>2$.  Then any exception must satisfy either ($m_1=m_3$, $m_2=m_4$)
\emph{or}
($m_1=m_4$, $m_2=m_3$).  Thus the number of exceptions is $2h^2-h$.  The result when $n>2$ now follows.
If $n=2$, then each $m_i$ is paired up with some other $m_j$.  Hence the number of exceptions is $3h^2-2h$.
Finally, we remark that the $2r-1$ in the estimate (\ref{weilbound}) can be improved to $2(r-1)$ when $n=2$.
This is because the genus of the curve $y^2=(x+m_1)\dots(x+m_{2r})$ is at most $r-1$.
\end{proof}

Applying Lemma~\ref{stirling} and comparing cases in (\ref{booze}) we arrive at the following result.
\begin{prop}\label{prop2}
Let $\chi$ be a non-principal Dirichlet character modulo $p$.
Let $r,h\in\Z^+$.  Then 
\[
  S_\chi(p,h,r)\leq p^{1/2}h^{2r}
  \cdot
    \begin{cases}
   \sqrt{2}\left(\frac{2r}{eh}\right)^rp^{1/2} + (2r-1) & r\geq 1\\
   3\left(1+\frac{p^{1/2}}{h^2}\right) & r=2\,.
   \end{cases}  
\]
\end{prop}

\subsection{A collection of intervals}

The collection of intervals given in the following proposition is a variation on those that
appear in a 1957 paper of Burgess (see~\cite{Burgess57}) and are standard
in the study of explicit bounds for character non-residues
(see, for example,~\cite{Norton71, Norton73, McGown2, TrevinoJNT}).
Explicit upper and lower bounds on the number of integer points
in these intervals will be essential for our results.

\begin{prop}\label{intervals}
Let $H\in(0,p)$ be a real number, $h\geq 2$ be an integer,
and set $X=H/h$.  For $0\leq t<q\leq X$, $(t,q)=1$,
define the intervals
\begin{align*}
\mathcal{I}(q,t)&=
\left(
\frac{tp}{q},\,\frac{tp+H}{q}-h+1
\right]
\,,
\\
\mathcal{J}(q,t)&=
\left[
\frac{tp-H}{q},\,\frac{tp}{q}-h+1
\right)
\,.
\end{align*}
Suppose $X\geq 2$ and $2HX<p$.
\begin{enumerate}
\item
The intervals
$\mathcal{I}(q,t)$, $\mathcal{J}(q,t)$
are disjoint subsets of $[-H,p-H)$.
\item 
Suppose $0\leq n\leq h-1$.
If $z\in\mathcal{I}(q,t)$ then $q(z+n)-pt\in(0,H]$,
and if $z\in\mathcal{J}(q,t)$ then $q(z+n)-pt\in[-H,0)$.
\item
The number of integer points in all the intervals
\[
N(X)=
\sum_{\substack{0\leq t<q\leq X\\(t,q)=1}}\sum_{z\in\mathcal{I}(q,t)\cup\mathcal{J}(q,t)}1
\]
satisfies
\[
  A(X)\frac{6}{\pi^2}X^2h\leq N(X)\leq B(X)\frac{6}{\pi^2}X^2h
\]
where
\begin{equation}\label{duke}
  A(X)=\left(1-\frac{2\pi^2}{9X}\right)
  \,,\quad
  B(X)=\left(1+\frac{2\pi^2}{9X}+\frac{1}{h}+\frac{\pi^2}{3h}\frac{\log X}{X}\right)
  \,.
\end{equation}
\end{enumerate}
\end{prop}

\begin{proof}
Using the assumption $2HX<p$, one can show that for $0\leq t<q\leq X$, $(t,q)=1$,
the intervals
\[
\left[(tp-H)q^{-1}, (tp+H)q^{-1}\right]
\]
are disjoint subsets of $[-H, p-H)$.
(See~\cite{McGown2} or~\cite{TrevinoJNT} for more details.)
The first claim of the proposition holds;
the second is true by definition.
We turn to the third claim.
Each interval above contains at least $H/q-h$ integers,
and at most $H/q-h+1$ integers.
For the lower bound we have
\[
N(X)\geq
2\sum_{\substack{0\leq t<q\leq X\\(t,q)=1}}
\left(\frac{H}{q}-h\right)
=
2h\left(X\sum_{1\leq q\leq X}\frac{\phi(q)}{q}-\sum_{1\leq q\leq X}\phi(q)\right)
\,,
\]
and for the upper bound we have
\[
N(X)\leq 2\sum_{\substack{0\leq t<q\leq X\\(t,q)=1}}
\left(\frac{H}{q}-h+1\right)
=
2h\left(X\sum_{1\leq q\leq X}\frac{\phi(q)}{q}-\sum_{1\leq q\leq X}\phi(q)\right)
+2\sum_{1\leq q\leq X}\phi(q)
\,.
\]
For ease of notation, write
\[
  S=X \sum_{q\leq X} \frac{\phi(q)}{q} - \sum_{q \leq X} \phi(q)
  \,,\;
  T= \sum_{q \leq X} \phi(q)
  \,,
\]
so that
\[
  2hS\leq N(X)\leq 2hS+2T
  \,.
\]
We now estimate $S$ and $T$ using the method in the proof of Lemma~3.1 of~\cite{TrevinoJNT}.
We can transform the series in an elementary way, arriving, as in \cite[p.\ 208]{TrevinoJNT} at
\begin{equation}\label{meat}
S = \frac{3}{\pi^{2}} X^{2} + \vartheta\frac{X^{2}}{2} \sum_{d>X} \frac{\mu(d)}{d^{2}} + \vartheta \frac{X}{2} \sum_{d\leq X} \frac{\mu(d)}{d} + \vartheta \frac{1}{8} \sum_{d\leq X} \mu^{2}(d),
\end{equation}
where we write $f(X) = \vartheta g(X)$ to denote $|f(X)| \leq |g(X)|$. We now aim at bounding each of the sums in (\ref{meat}), which we denote by $S_1$, $S_2$, $S_3$.




We have $|S_{1}| \leq X^{-1}$ by Claim 3.1 in \cite{TrevinoJNT}. This could be improved, but will suffice for our purposes.

The bound $|S_{2}| \leq 2/3 + 3/X$ appears in Claim 3.3 in \cite{TrevinoJNT}. We can make a cheap improvement courtesy of a result of Ramar\'{e} \cite{Ram} that gives $|S_{2}| \leq 1/10$ for all $X\geq 7$. Extending this via a very quick check, we find that
$$|S_{2}| \leq \frac{1}{10} + \frac{2}{X}, \quad (X \geq 1).$$


To estimate $S_{3}$ we use directly Lemma 4.2 in Cipu \cite{Cipu}, which gives
\begin{equation}\label{E:cipu}
|S_{3}|\leq 6\pi^{-2}X + 0.679091\sqrt{X}, \quad (X\geq 1)\,.
\end{equation}
Putting this together we have
\begin{equation}\label{binge} 
S- \frac{3}{\pi^{2}}X^{2} = \vartheta\left\{ X\left( \frac{1}{2} + \frac{1}{20} + \frac{3}{4\pi^{2}}\right) + 1 + \frac{0.68 \sqrt{X}}{8}\right\}.
\end{equation}

From (\ref{binge}) it is easy to show that $|S - 3\pi^{-2} X^{2}|\leq (2/3) X$ for all $X\geq 38$. A quick computational check establishes that this is also true for $1\leq X < 38$.

We can play the same game with $T$; namely, we have
$$T = \frac{3}{\pi^{2}}X^2 -\frac{X^2}{2} \sum_{d>X} \frac{\mu(d)}{d^{2}} + \frac{X}{2} \sum_{d\leq X} \frac{\mu(d)}{d} - X \sum_{d\leq X} \frac{\mu(d)}{d} \left\{ \frac{X}{d}\right\} + \frac{1}{2} \sum_{d\leq X} \mu(d) \left( \left\{ \frac{X}{d} \right\}^{2} - \left\{ \frac{X}{d}\right\}\right).$$
Call these sums $T_{1}, T_{2}, T_{3}, T_{4}$. All but $T_{3}$ are estimated as before. For $T_{3}$ we have
$$|T_{3}| \leq \sum_{d\leq X} \frac{\mu^{2}(d)}{d}.$$
Trevi\~{n}o \cite{TrevinoJNT} gives a bound on this between his equations (15) and (16).
We can do a little better with~(\ref{E:cipu}) using partial summation;
namely $|T_3|\leq 6\pi^{-2}\log X+2$.
Putting all this together we have
$$ T- \frac{3}{\pi^{2}}X^{2} = \vartheta\left\{ \frac{6}{\pi^{2}} X \log X + X\left( \frac{1}{2} + \frac{1}{20} + 2 + \frac{3}{4\pi^{2}}\right) + \frac{0.68\sqrt{X}}{8} + 1 \right\}.$$

We find that $|T- 3\pi^{-2}X^2| \leq X \log X$ for $X\geq 1000$. A quick computational check establishes this inequality for all $X\geq 2$.
This establishes (\ref{duke}) and proves the proposition.
\end{proof}

\subsection{The sieve}

We will make use of the same sieve employed in~\cite{Cohen}
in the form of the following result.

\begin{prop}\label{propsieve}
Let $e$ be an even divisor of $p-1$ and
let $p_1,\dots,p_s$ denote the primes dividing $p-1$ that do not divide $e$.
Set $\delta=1-\sum_{i=1}^sp_i^{-1}$.  Assume $\delta>0$.
Then we have
\begin{eqnarray*}
\frac{f(n)}{\delta\theta(e)}
\geq
1
+\frac{1}{\delta}\sum_{i=1}^s
\theta(p_i)
\sum_{\substack{d|e}}
\frac{\mu(p_i d)}{\phi(p_i d)}
\sum_{\substack{\chi\\ \ord(\chi) = p_{i}d}}
\chi(n)
+
\sum_{\substack{d|e\\d>1}}
\frac{\mu(d)}{\phi(d)}
\sum_{\substack{\chi\\ \ord(\chi) = d}}
\chi(n)
\,.
\end{eqnarray*}
\end{prop}

\begin{proof}
We say that $n$ is $e$-free if
the equation $y^d\equiv n\pmod{p}$ is insoluble for all divisors $d$ of $e$ with $d>1$.  An integer is a primitive root
if and only if it is $(p-1)$-free.
We define the function
$$
 f_e(n)=
 \begin{cases}
 1 &\text{ if $n$ is $e$-free}\\
 0 & \text{ otherwise}.\\
 \end{cases}
$$
We have the following equation
$$
 f_e(n)/\theta(e)=
 1+\sum_{\substack{d|e\\d>1}}
 \frac{\mu(d)}{\phi(d)}
 \sum_{\substack{\chi\\ \ord(\chi) = d}}
 \chi(n).
$$
One verifies that
\[
  f_{p-1}(n)\geq
  \sum_{i=1}^s(f_{p_i e}(n)-\theta(p_i)f_e(n))
  +\delta f_e(n)
\]
and
\[
  f_{p_i e}(n)-\theta(p_i)f_e(n)
  =
  \theta(p_i e)\sum_{d|e}
  \frac{\mu(p_i d)}{\phi(p_i d)}
  \sum_{\ord(\chi)=p_i d}
  \chi(n)
\]
which leads to
the desired conclusion.
See~\cite{MTT} for the details.
\end{proof}




\section{Main theorem}\label{queen}
We now come to our main result, from which Theorem \ref{T:1} follows.
Given $r,h\geq 1$,
suppose that for all non-principal characters modulo~$p$
we have
\[
\sum_{x\in\F_p} \left|\sum_{n=0}^{h-1} \chi(x+n)\right|^{2r}\leq
W(p,h,r)p^{1/2}h^{2r}
\,.
\]

\begin{theorem}\label{T:3}
Let $p$ be an odd prime.
Let $e$ be an even divisor of $p-1$ and let $p_1,\dots,p_s$ be the primes dividing $p-1$ that do not divide $e$.
Set $\delta=1-\sum_{i=1}^sp_i^{-1}$.  Suppose $\delta>0$.
Let $r,h\in\Z^+$ and $H>0$.
Suppose $h\geq 2$, $H\geq 2h$, and $2H^2<hp$.
Set $X=H/h$.
If
\[
  \frac{\pi^2}{6}\frac{B(X)^{2r-1}}{A(X)^{2r}}
  \left(\left(2+\frac{s-1}{\delta}\right)2^{\omega-s}\right)^{2r}hp^{1/2}W(p,h,r)<H^2
  \,,
\]
then $g(p)<H$.
\end{theorem}

\begin{proof}
Suppose there are no primitive roots in the interval $(0,H]$.
We aim to create a contradiction.
By our hypothesis and Proposition~\ref{intervals}, for all
$z\in\mathcal{I}(q,t)$ and $0\leq n<h$ we have
$q(z+n)-pt\in(0,H]$ and hence $f(q(z+n))=0$;
similarly, for all
$z\in\mathcal{J}(q,t)$ and $0\leq n<h$, we have
$q(z+n)-pt\in[-H,0)$ and hence $f(-q(z+n))=0$.
Therefore, by Proposition~\ref{propsieve}, we have
\[
\frac{1}{\delta}\sum_{i=1}^s
\theta(p_i)
\sum_{\substack{d|e}}
\frac{\mu(p_i d)}{\phi(p_i d)}
\sum_{\substack{\chi\\ \ord(\chi) = p_{i}d}}
\chi(\pm q)
\chi(z+n)
+
\sum_{\substack{d|e\\d>1}}
\frac{\mu(d)}{\phi(d)}
\sum_{\substack{\chi\\ \ord(\chi) = d}}
\chi(\pm q)
\chi(z+n)
\leq -1.
\]
Summing this over $q,t,z,n$ we find that if we define
\begin{align*}
\mathcal{S}:=
&
\frac{1}{\delta}\sum_{i=1}^s
\theta(p_i)
\sum_{\substack{d|e}}
\frac{\mu(p_i d)}{\phi(p_i d)}
\sum_{\substack{\chi\\ \ord(\chi) = p_{i}d}}
\sum_{\substack{0\leq t<q\leq X\\(t,q)=1}}
\sum_{z\in\mathcal{I}(q,t)\cup\mathcal{J}(q,t)}
\chi(\pm q)
\sum_{n=0}^{h-1}
\chi(z+n)
\\
&
+
\sum_{\substack{d|e\\d>1}}
\frac{\mu(d)}{\phi(d)}
\sum_{\substack{\chi\\ \ord(\chi) = d}}
\sum_{\substack{0\leq t<q\leq X\\(t,q)=1}}
\sum_{z\in\mathcal{I}(q,t)\cup\mathcal{J}(q,t)}
\chi(\pm q)
\sum_{n=0}^{h-1}
\chi(z+n),
\end{align*}
we have
\[
  \mathcal{S}\leq -\sum_{\substack{0\leq t<q\leq X\\(t,q)=1}}
\sum_{z\in\mathcal{I}(q,t)\cup\mathcal{J}(q,t)}\sum_{n=0}^{h-1}
1,
\]
which implies
$|\mathcal{S}|\geq A(X)(6/\pi^2)X^2h^2$.
The goal is to give a sufficiently strong upper bound on $\mathcal{S}$
so as to create the desired contradiction.
We estimate
\begin{align*}
    |\mathcal{S}|
    &\leq
\frac{1}{\delta}\sum_{i=1}^s
\theta(p_i)
\sum_{\substack{d|e}}^\flat
\sum_{\substack{0\leq t<q\leq X\\(t,q)=1}}
\sum_{z\in\mathcal{I}(q,t)\cup\mathcal{J}(q,t)}
\max_{\ord(\chi)=p_i d}\left|\sum_{n=0}^{h-1}\chi(z+n)\right|
\\
&
\quad
+
\sum_{\substack{d|e\\d>1}}^\flat
\sum_{\substack{0\leq t<q\leq X\\(t,q)=1}}
\sum_{z\in\mathcal{I}(q,t)\cup\mathcal{J}(q,t)}
\max_{\ord(\chi)=d}\left|\sum_{n=0}^{h-1}\chi(z+n)\right|
\\
&
\leq
\left(1+\frac{1}{\delta}\sum_{i=1}^s\theta(p_i)\right)\sum_{d|e}^\flat
\sum_{\substack{0\leq t<q\leq X\\(t,q)=1}}
\sum_{z\in\mathcal{I}(q,t)\cup\mathcal{J}(q,t)}
\max_{\ord(\chi)\neq 1}\left|\sum_{n=0}^{h-1}\chi(z+n)\right|
\\
&
\leq\left(2+\frac{s-1}{\delta}\right)2^{\omega-s}
\sum_{\substack{0\leq t<q\leq X\\(t,q)=1}}
\sum_{z\in\mathcal{I}(q,t)\cup\mathcal{J}(q,t)}
\max_{\ord(\chi)\neq 1}\left|\sum_{n=0}^{h-1}\chi(z+n)\right|
\,.
\end{align*}

Using H\"older's inequality, we find
\begin{align*}
    |\mathcal{S}|\leq
    \left(2+\frac{s-1}{\delta}\right)2^{\omega-s}
&\left(
\sum_{\substack{0\leq t<q\leq X\\(t,q)=1}}
\sum_{z\in\mathcal{I}(q,t)\cup\mathcal{J}(q,t)}
\max_{\ord(\chi)\neq 1}\left|\sum_{n=0}^{h-1}\chi(z+n)\right|^{2r}
\right)^{\frac{1}{2r}}
\\
&
\times
    \left(
\sum_{\substack{0\leq t<q\leq X\\(t,q)=1}}
\sum_{z\in\mathcal{I}(q,t)\cup\mathcal{J}(q,t)}
1
\right)^{1-\frac{1}{2r}}.
\end{align*}
Using the fact that the intervals 
in question are disjoint, we can complete the character sum
and invoke the Weil bound.  This yields
\begin{align*}
\sum_{\substack{0\leq t<q\leq X\\(t,q)=1}}
\sum_{z\in\mathcal{I}(q,t)\cup\mathcal{J}(q,t)}
\left|
\sum_{n=0}^{h-1}
\chi(z+n)
\right|^{2r}
\leq
\sum_{x\in\mathbb{F}_p}
\left|
\sum_{n=0}^{h-1}
\chi(x+n)
\right|^{2r}
\leq
W(p,h,r)p^{1/2}h^{2r}.
\end{align*}
Dealing with the second factor, we have
$\sum_{t,q}
\sum_{z}
1
\leq
B(X)(6/\pi^2)X^2 h
$.
Putting this together, we have shown that
\begin{align*}
  |\mathcal{S}|
  &\leq
  \left(2+\frac{s-1}{\delta}\right)2^{\omega-s}
  \left(W(p,h,r)p^{1/2}h^{2r}\right)^{\frac{1}{2r}}
  \left(B(X)\frac{6}{\pi^2}X^2h\right)^{1-\frac{1}{2r}}.
  \end{align*}
Recall that we need to show the above is less than $A(X)(6/\pi^2)X^2h^2$.
After raising everything to power of $2r$ and simplifying, we find
the following condition suffices:
\[
  B(X)^{2r-1}\left(\left(2+\frac{s-1}{\delta}\right)2^{\omega-s}\right)^{2r}
  W(p,h,r)p^{1/2}
  <
  A(X)^{2r}\frac{6}{\pi^2}X^2 h.  
\]
Substituting $X=H/h$ and isolating the $H^2$ term yields
the condition in the statement of the theorem.
\end{proof}

\section{Consequences of Theorem \ref{T:3}}\label{prince}

First we establish the two corollaries stated in \S\ref{intro}.

\begin{proof}[Proof of Corollary~\ref{T:2}]
Set $H=p^{5/8}$ and $h=\lceil 2p^{1/4}\rceil$
so that  $2H^2<hp$.
We will apply Theorem \ref{T:3} with $r=2$.
Assuming $p\geq 10^{20}$,
we have $X=H/h\geq 10^7$ and $h\geq 2\cdot 10^5$.
We have
  \[  
    W(p,h,2)\leq 3\left(1+\frac{p^{1/2}}{h^2}\right)\leq\frac{15}{4}
    \,,
  \]
  and one verifies that $A(X)\geq 1-10^{-6}$ and $B(X)\leq 1+10^{-5}$.
  Consequently, one finds that $g(p)<p^{5/8}$ provided
  \begin{equation}\label{test}
    13\left(\left(2+\frac{s-1}{\delta}\right)2^{\omega-s}\right)^4<p^{1/2}
    \,.
  \end{equation}
  
  When $\omega\leq 8$, condition (\ref{test}) holds trivially using $s=0$ (and hence $\delta=1$).
  When $9\leq\omega\leq 17$, we set $s=\omega-3$ and note that
  $\delta\geq 1-\sum_{i=3}^\omega q_i^{-1}$ where $q_i$ denotes the $i$-th prime;
  in this case, one verifies that
  the left-hand side of (\ref{test}) is less than $10^{11}\leq p^{1/2}$.
  When $18\leq\omega\leq 50$, we again set $s=\omega-3$ and combine
  with the fact that $p>\prod_{i=1}^\omega q_i$ to verify that (\ref{test}) holds.
  
  When $p\geq 10^{1000}$, we can use the bound
  $\omega(p-1)\leq 1.39\log p/\log\log p$ (see~\cite{Robin}) to verify that (\ref{test}) holds
  when $s=0$.  Hence we may assume that $p\leq 10^{1000}$ which implies that
  $\omega\leq 200$.  For the remaining range $50<\omega<200$,
  the choice of $s=\omega-5$ works.
\end{proof}

\begin{proof}[Proof of Corollary~\ref{lonely}]
Set $H=(0.999)p^{1/2}$.  Set $h=\lceil p^{1/4}\rceil$.  The condition
from Theorem~\ref{T:3} becomes 
\[
  7\left(\left(2+\frac{s-1}{\delta}\right)2^{\omega-s}\right)^4 < p^{1/4}\,.
\]
The result now follows from an analysis similar to the proof of Corollary \ref{T:2}.
\end{proof}

In light of Corollary~\ref{lonely}, to prove Theorem~\ref{T:1} it suffices to establish the following intermediary result.
We have chosen to record this result separately as it could be used to improve the range of $p$ in
which Theorem~\ref{T:1} holds.

\begin{theorem}\label{T:Win}
Suppose $p\geq 10^{15}$.  Let $r\geq 2$ be an integer. 
If
\[
g(p)<p^{\frac{1}{2}+\frac{1}{4r}}
\,,
\]
then
\[
  g(p)<2r\, 2^{r\omega(p-1)} p^{\frac{1}{4}+\frac{1}{4r}}
  \,.
\]
\end{theorem}

\begin{proof}
Without loss of generality, we may assume $(8\log 2)r<\log p$;
indeed, if this is not true, then one finds
$g(p)<p^{\frac{1}{2}+\frac{1}{4r}}\leq 2^{r\omega}p^{\frac{1}{4}+\frac{1}{4r}}$.
It now follows from this that
$p^{\frac{1}{2r}}\geq 16$.

Set $H=Cr2^{r\omega}p^{\frac{1}{4}+\frac{1}{4r}}$ with $C=2$.
We may assume $g(p)>H$; otherwise, there is nothing to prove.
We will invoke Theorem \ref{T:3} with $e=p-1$ and $s=0$.
We choose
\[
  h:=\left\lceil \frac{2r}{e}(2p)^{\frac{1}{2r}}\left(\frac{r-1}{2r-1}\right)^{\frac{1}{r}}\right\rceil.
\]
Notice that
$h\geq 33$, 
$h\geq \frac{r}{2}\,p^{\frac{1}{2r}}$,
and
\[
   h\leq(1.031) \frac{2r}{e}p^{\frac{1}{2r}}\left(\frac{\sqrt{2}(r-1)}{2r-1}\right)^{\frac{1}{r}}\leq rp^{\frac{1}{2r}}
  \,.
\]
Setting $X=H/h$, we have the estimate
\[
  \frac{2H^2}{h}\leq e C^2
  \left(\frac{2r-1}{\sqrt{2}(r-1)}\right)^{\frac{1}{r}}
  r \left(2^{\omega}-1\right)^{2r}p^{1/2},
\]
whence to verify $2HX<p$, it suffices to prove
\begin{equation}\label{E:1}
e C^2 r   \left(\frac{2r-1}{\sqrt{2}(r-1)}\right)^{\frac{1}{r}} 2^{2r\omega}<p^{1/2}
\,.
\end{equation}
If (\ref{E:1}) fails, then
we have
$Cr 2^{r\omega}>\sqrt{\frac{r}{e}} \left(\frac{\sqrt{2}(r-1)}{2r-1}\right)^{\frac{1}{2r}}p^{1/4}$
and hence
\[
  g(p)>\sqrt{\frac{r}{e}} \left(\frac{\sqrt{2}(r-1)}{2r-1}\right)^{\frac{1}{2r}}p^{\frac{1}{2}+\frac{1}{4r}}>\frac{\sqrt{r}}{2}p^{\frac{1}{2}+\frac{1}{4r}}
  \,,
\]
which would contradict our hypothesis.
We record the estimate
\begin{equation}\label{record}
  X=\frac{H}{h}\geq C\,2^{r\omega}\,p^{\frac{1}{4}-\frac{1}{4r}} \geq 2000
  \,.
\end{equation}
For the Weil bound in Proposition \ref{prop2}, we have
\begin{align*}
W(p,h,r)
=
\left(\sqrt{2}\left(\frac{2r}{eh}\right)^rp^{1/2}+2r-1\right)
\leq
r(2r-1)(r-1)^{-1}.
\end{align*}

We bound the quantities $A(X)$ and $B(X)$ appearing in Proposition~\ref{intervals}.
Since $2^{r\omega}/r\geq 8$, we have $X/r\geq 8Cp^{1/8}$ and therefore,
using Bernoulli's inequality, we find
$$
  A(X)^r\geq 1-\frac{2r\pi^2}{9X}
  \geq
  1-\frac{\pi^2}{72\,p^{\frac{1}{8}}}
  \geq 
  0.998
  \,.
$$
We bound $B(X)$ in a similar way.
Writing $X=1+Y$, we use the inequality $(1+Y)^r\leq 1+rY+r^2Y^2$ that holds when $rY\leq 0.35$.
We have
\[
  rY=\frac{2\pi^2r}{9X}+\frac{r}{h}+\frac{\pi^2r}{3h}\frac{\log X}{X}
\]
and using $r/h\leq 2/p^{\frac{1}{2r}}\leq 1/8$ we find $rY\leq 0.129$.
Hence $B(X)^r\leq 1.145$.

After verifying that
\begin{equation}\label{finale}
  \frac{\pi^2}{6}\frac{1.145^2}{0.998^2}\frac{2}{e}(1.031)2^{\frac{1}{2r}}\left(\frac{2r-1}{r-1}\right)^{1-\frac{1}{r}}
  < 4
\end{equation}
the condition in the statement of Theorem~\ref{T:3} reduces to
\[
  4 r^2(2^\omega)^{2r}p^{\frac{1}{2}+\frac{1}{2r}}\leq H^2
  \,,
\]
which is true given our definition of $H$.
\end{proof}

The following is our final result from which
Theorem~\ref{Odd fellow} follows.

\begin{theorem}\label{T:Win2}
Suppose $p\geq 10^{15}$.  Let $r\geq 2$ be an integer. 
Let $e$ be an even divisor of $p-1$ and let $p_1,\dots,p_s$ be the primes dividing $p-1$ that do not divide $e$.
Set $\delta=1-\sum_{i=1}^sp_i^{-1}$.  If $\delta>0$ and
\[
g(p)<p^{\frac{1}{2}+\frac{1}{4r}}
\,,
\]
then we have
\[
  g(p)<2r\, \left(\left(2+\frac{s-1}{\delta}\right)2^{\omega-s}\right)^rp^{\frac{1}{4}-\frac{1}{4r}}
  \,.
\]
\end{theorem}

\begin{proof}
We proceed almost exactly as in the proof of Theorem~\ref{T:Win}
except that we set $H=Cr((2+(s-1)\delta^{-1})2^{\omega-s})^rp^{\frac{1}{4}+\frac{1}{4r}}$.
The estimate~(\ref{record}) becomes instead
\[
  X=\frac{H}{h}\geq C\left(\left(2+\frac{s-1}{\delta}\right)2^{\omega-s}\right)^rp^{\frac{1}{4}-\frac{1}{4r}}
  \geq 500
  \,.
\]
The bounds for $A(X)$ and $B(X)$ change slightly.  In this case we have
$X/r\geq 2Cp^{1/8}$ which leads to $A(X)^2\geq 0.992$,
and $rY\leq 0.138$ which leads to $B(X)^r\leq 1.158$.  However, the expression
that appears on the lefthand side of~(\ref{finale}) adjusted appropriately is still less than~$4$.
\end{proof}

\subsection*{Acknowledgements}
This work was conceived when the first author visited the second author: the authors wish to thank the School of Science at UNSW Canberra at ADFA and the Rector's Visiting Fellowship program for supporting the visit.

%
%
%


\end{document}